\documentclass{article}
\usepackage{amsmath, amssymb, amsthm, epsf, graphicx, amscd, amsfonts, amscd, multirow}

\usepackage[pagewise, displaymath, mathlines]{lineno}

\usepackage{setspace}
\newtheorem{theorem}{Theorem}[section]

\newtheorem{corollary}[theorem]{Corollary}

\def\bb #1{ {\mathbb #1} }
\def\c #1{ {\mathcal #1} }

\begin{document}
\title{$p$-adic unit roots of $L$-functions over finite fields}

\author{C. Douglas Haessig}

\maketitle

\abstract{In this brief note, we consider $p$-adic unit roots or poles of $L$-functions of exponential sums defined over finite fields. In particular, we look at the number of unit roots or poles, and a congruence relation on the units. This raises a question in arithmetic mirror symmetry.}


\section{Introduction}

$L$-functions of exponential sums over finite fields are rational functions. The following note considers two questions concerning the $p$-adic unit roots and poles of these rational functions. First, how many $p$-adic unit roots and poles are there? Second, what type of geometric relation do the unit roots satisfy; more specifically, is there a congruence relation the unit roots satisfy? 

To motivate, consider the affine Legendre family of elliptic curves $E_{\bar \lambda}$ defined by $y^2 = x(x-1)(x-\bar \lambda)$, where $\bar \lambda \in \bb F_p \setminus \{0, 1\}$. Setting
\[
\#E_{\bar \lambda}(\bb F_{p^m}) := \# \{ (\bar x, \bar y) \in \bb F_{p^m}^2 \mid \bar y^2 = \bar x(\bar x-1)(\bar x- \bar \lambda) \},
\]
then it is well-known that the associated zeta function takes the form
\[
Z(E_{\bar \lambda} / \bb F_p, T) := \exp\left( \sum_{m = 1}^\infty \#E_{\bar \lambda}(\bb F_{p^m}) \frac{T^m}{m} \right)
= \frac{(1- \pi_0(\bar \lambda) T)(1 - \pi_1(\bar \lambda)T)}{(1 - pT)}.
\]
Denote by $H(\lambda)$ the series obtained from the hypergeometric series ${}_2 F_1(\frac{1}{2}, \frac{1}{2}; 1; \lambda)$  by taking the terms of $\lambda$ up to degree $(p-1)/2$. If $H(\bar \lambda) \not= 0$, then there is a unique $p$-adic unit root $\pi_0(\bar \lambda)$ of the zeta function, and
\begin{equation}\label{E: formula}
\pi_0(\bar \lambda) \equiv (-1)^{(p-1)/2} H(\bar \lambda) \quad \text{mod } p.
\end{equation}
In this example, the number of unit roots is bounded by the genus 1. For hyperelliptic curves of genus $g$, the number of $p$-adic unit roots is bounded by the genus $g$. The geometric interpretation of (\ref{E: formula}) comes from noting that the hypergeometric series ${}_2 F_1(\frac{1}{2}, \frac{1}{2}; 1; \lambda)$ is the unique holomorphic solution of the associated Picard-Fuchs equation of the Legendre family. Formulas for other families have also been recently studied; see for instance \cite{MR2966711} and \cite{MR2448441}.

In the following, we focus on $L$-functions of exponential sums over a finite field, defined over hypersurfaces in the algebraic torus $\bb G_m^n$ or affine space. The main results are Theorems \ref{T: rational}, \ref{T: main_toric}, and \ref{T: main_affine}. Section \ref{S: one} considers rational functions in one-variable, and it is shown that in many cases which arise in practice, the unit roots are in fact 1-units, meaning they satisfy $|x-1|_p < 1$. Since this is not always the case, the result is then generalized to all rational functions. Sections \ref{S: many} and \ref{S: many_affine} consider $L$-functions in many variables defined over affine hypersurfaces in $\bb G_m^n$ and $\bb A^n$, respectively. Lastly, motivated from the result in Section \ref{S: many}, Section \ref{S: mirror} looks at the question of comparing the $L$-functions of exponential sums of regular functions defined on the complement of two mirror varieties. 

We note that similar questions may be asked for other types of zeta functions, such as  Dwork's unit root $L$-function \cite{Wan-DworkConjectureunit-1999} \cite{Wan-Rankonecase-2000} \cite{Wan-Higherrankcase-2000} \cite{MR3249829} or $L$-functions over rings, such as \cite{MR2516429}. This has already been answered in the case of the zeta function of divisors \cite{MR2516980} \cite{MR2029511} \cite{Wan-ZetaFunctionsof-1992}, and so it would be interesting to study similar questions for the function field height zeta function \cite{MR2423749} \cite{MR1196533}.

{\bf Acknowledgments.} I wish to thank Steven Sperber for his encouragement and suggestions, and Alan Adolphson for his comments and corrections. This work was partially supported by the Simons Foundation.

\section{Rational functions in one variable}\label{S: one}

Let $\bb F_q$ be the finite field with $q$ elements and characteristic $p$, and $\psi$ a non-trivial additive character of $\bb F_q$. Set $\psi_m := \psi \circ \text{Tr}_{\bb F_{q^m} / \bb F_q}$. Let $f/g \in \bb F_q(x)$ be a rational function. Throughout this section, we will assume $f/g$ has a pole at zero and infinity (meaning, $g(0) = 0$ and $deg(f) > deg(g)$), and the order of every pole is not divisible by $p$. For each $m \geq 1$, define the exponential sum
\[
S_m(f/g) := \sum_{\bar x \in \bb F_{q^m}, g(\bar x) \not= 0 } \psi_m( f(\bar x) / g(\bar x) ).
\]
It follows from \cite{MR773094} (or \cite{MR2067435}) that the associated $L$-function
\[
L(f/g, T) := \exp\left(\sum_{m =1}^\infty S_m(f/g) \frac{T^m}{m} \right)
\]
is a polynomial over $\bb Z[\zeta_p]$ of degree $deg(f) + n - 1$ with precisely $n$ $p$-adic unit roots, where $\zeta_p$ is a primitive $p$-th root of unity. When the roots of $g$ lie in the base field $\bb F_q$, then each of the units of the $L$-function must be a \emph{1-unit}, or \emph{principal unit}, as we will now show; this means they satisfy $| x - 1|_p < 1$.

\begin{theorem}\label{T: rational}
Assuming $p > n$ and that every root of $g$ lies in $\bb F_q$, then each $p$-adic unit root of $L(f/g, T)$ is a 1-unit.
\end{theorem}

\begin{proof}
Let $\eta_1, \ldots, \eta_n$ denote the unit roots of $L(f/g, T)$. Let $\tilde \pi$ be the uniformizer of the ring of integers of the field obtained by adjoining the zeros of $L(f/g, T)$ to $\bb Q_p(\zeta_p)$. Since $L(f/g, T) = \prod (1 - \eta_j T)$ mod($\tilde \pi$), we see that $S_m(f/g) \equiv - \sum \eta_j^m$ mod $\tilde \pi$. On the other hand, since $\zeta_p \equiv 1$ mod $\tilde \pi$, we have $S_m(f/g) \equiv -n$ mod $\pi$. Consequently, the power sums $p_m := \sum \eta_j^m \equiv n$ mod $\tilde \pi$ for every $m$.

Consider now the product $\prod_{j=1}^n (1 - \eta_j) = \sum_{k=0}^n (-1)^k e_k$, where $e_k$ is the $k$-th elementary symmetric polynomial in $\eta_1, \ldots, \eta_n$. Newton's identity relates $e_k$ and the power sums $p_k$:
\[
e_1 = p_1 \quad \text{and} \quad k e_k = \sum_{i=1}^k (-1)^{i+1} e_{k-i} p_i.
\]
In our case, since $p > n$ and $p_m \equiv n$ mod $\tilde \pi$, we see that $e_k \equiv \binom{n}{k}$ for each $k = 0, 1, \ldots, n$. Hence
\[
\prod_{j=1}^n (1 - \eta_j) \equiv \sum_{k=0}^n (-1)^k \binom{n}{k} = (1 - 1)^k = 0 \quad \text{mod } \tilde \pi,
\]
implying at least one of the unit roots must be a 1-unit. If we suppose $\eta_1 \equiv 1$ mod $\tilde \pi$, then we get the new power sum identity $\sum_{j=2}^n \eta_j^m \equiv n-1$ mod $\tilde \pi$. Repeating the above argument obtains the result.
\end{proof}

We note that, when the roots of $g$ lie in the base field $\bb F_q$, then Zhu \cite{Zhu-$L$-functionsofexponential-2004} has shown that the other roots have slopes lying on or above the Hodge polygon.

While unit roots of $L$-functions need not always be 1-units,  a similar statement to Theorem \ref{T: rational} for general one-variable rational functions does hold. To state this, define $Z(g, T) := \exp \left(\sum N_m \frac{T^m}{m} \right)$, the zeta function associated with the point count
\[
N_m := \# \{ \bar x \in \bb F_{q^m} \mid g(\bar x) = 0\}.
\]

\begin{theorem}\label{T: irred}
Let $d_i$ be the degrees of the irreducible polynomials factoring $g$ over $\bb F_q$. If  $p > deg(f) + n -1$, then
\[
L(f/g, T) \equiv \prod (1 - T^{d_i}) \qquad \text{mod } \pi,
\]
where $\pi := 1 - \zeta_p$.
\end{theorem}

\begin{proof}
Factoring $g$ into irreducibles over $\bb F_q$ of degrees $d_i$, note that $Z(g, T)^{-1} = \prod (1 - T^{d_i})$. Next, observe that $S_m(f/g) \equiv - N_m$ mod $\pi$ for every $m$. Since $p$ is greater than the degrees of $L(f/g, T)$ and $Z(g, T)^{-1}$, the result follows.
\end{proof}

\section{Complements of hypersurfaces in $\bb G_m$}\label{S: many}

We now generalize Theorem \ref{T: irred} to many variables. Let $f, g \in \bb F_q[x_1, \ldots, x_n]$ be polynomials of degree $D$ and $d$, respectively. 
Let $H_g^*$ denote the hypersurface in $\bb G_m^n$ defined by $g = 0$ over $\bb F_q$, and let $V_g^*$ denote the complement of $H_g^*$ in $\bb G_m^n$. For each positive integer $m$ define the exponential sum
\[
S_m(V_g^*, f) = \sum \psi_m(  f(x) ),
\]
where the sum runs over all $x \in V_g^*(\bb F_{q^m})$. The associated $L$-function
\[
L(V_g^*, f; T) := \exp \left( \sum_{m=1}^\infty S_m(V_g^*, f) \frac{T^m}{m} \right)
\]
is well-known to be a rational function defined over the cyclotomic field $\bb Q(\zeta_p)$, where $\zeta_p$ is a primitive $p$-th root of unity. Denote by $\hat g$ the polynomial of obtained from $g$ by replacing the coefficients with their Teichm\"uller representative. Decompose $\hat g = \hat g_d + \hat g_{d-1} + \cdots + \hat g_0$ into homogeneous forms $\hat g_i$ of degree $i$. Throughout this section, we will assume that
\begin{equation}\label{E: Riech property}
\text{$\hat g_d$ is a product of distinct irreducible factors.}
\end{equation}
This condition is not required if $n = 1$.

Let $Z(H_g^*, T) := \exp \left(\sum N_m \frac{T^m}{m} \right)$ be the zeta function associated with the point count
\[
N_m := \# H_g^*(\bb F_{q^m}) := \# \{ x \in (\bb F_{q^m}^*)^n \mid g(x) = 0 \}.
\]
As a consequence of the Dwork trace formula, the zeta function has the following simple description mod $p$ (see \cite{MR1650604}). Define the $\bb F_q$-vector space
\[
R_d := \bb F_q[x_1, \ldots, x_n]_{\leq d}
\]
consisting of polynomials over $\bb F_q$ of degree at most $d$. Define the Cartier operator $\psi_q$ as follows. Writing $x^u := x_1^{u_1} \cdots x_n^{u_n}$, then
\[
\psi_q(x^u) := 
\begin{cases}
x^{u/q} & \text{if $q \mid u_i$ for every $i$} \\
0 & \text{otherwise.}
\end{cases}
\]
Observe that multiplication by $g^{q-1}$ is a map from $R_d$ into $R_{d q}$, and thus $\psi_q \circ g^{q-1}$ is an operator on $R_d$. Wan \cite{MR1650604} observes that
\begin{equation}\label{E: Wan zeta_toric}
Z(H_g^*, T)^{(-1)^n} \equiv (1 - T) \> det(1 - \psi_q \circ g^{q-1} T \mid R_d) \qquad \text{mod } p,
\end{equation}
a polynomial of degree at most $1 + \text{dim}_{\bb F_q} R_d = 1 + \binom{n+d}{n}$.

\begin{theorem}\label{T: main_toric}
Suppose (\ref{E: Riech property}) and $p > \max\{ \binom{n + d}{n}+1, D \}$. Then
\[
L(V_g^*, f; T) \equiv (1 - T)^{(-1)^n} \cdot Z(H_g^*, T)^{-1} \quad \text{mod}(\pi)
\]
where $\pi := 1 - \zeta_p$. Moreover,
\[
L(V_g^*, f; T)^{(-1)^{n+1}} \equiv det(1 - \psi_q \circ g^{q-1} T \mid R_d) \quad \text{mod}(\pi),
\]
a polynomial of degree at most $\binom{d+n}{n}$. In particular, the number of unit roots of $L(V_g^*, f; T)$ is bounded by $\binom{d+n}{n}$.
\end{theorem}

\begin{proof}
From the Reich trace formula \cite{MR0255549},
\[
L(V_g^*, f; T)^{(-1)^{n+1}} = det(1 - \alpha T)^{\delta^n} = \frac{\prod (1 - \omega_i T)}{\prod (1 - \eta_j T)},
\]
where $\alpha$ is a $p$-adic operator acting on a space of functions, and $\delta$ is the map defined by $h(T)^\delta := h(T) / h(qT)$. Hence, $L(V_g^*, f; T)^{(-1)^{n+1}} \equiv det(1 - \alpha T)$ mod $q$. We refine this further using the work of Adolphson-Sperber \cite{MR573096}, which says that the $q$-adic Newton polygon of $det(1 - \alpha T)$ lies on or above the lower convex hull of the points $(0, 0)$ and $(\sum_{k=0}^m W(k), D^{-1} \sum_{k=0}^m k W(k))$ for $m = 0, 1, 2, \ldots$. We refer the reader to \cite{MR573096} for the definition of $W(k)$ as it is complicated; however, we do note some properties. First, $W(k) \geq \binom{n+k-1}{k}$ for $k \geq 0$, and so $W(k) > 0$ for $k \geq 0$. Second $W(0) =  \binom{n+d}{n}$.

By hypothesis $p > D$, and so $ord_\pi(\cdot) \geq (p-1) ord_q(\cdot) \geq D \> ord_q(\cdot)$. Consequently, $det(1 - \alpha T)$ mod $\pi$ is a polynomial of degree at most $W(0)$. Next, since $\zeta_p \equiv 1$ mod $\pi$ and $\# V_g^*(\bb F_{q^m}) \equiv (-1)^n - \# H_g^*(\bb F_{q^m})$ mod $q$, we have
\begin{equation}\label{E: equiv}
(-1)^{n+1} S_m(V_g^*, f) \equiv (-1)^{n+1} \# V_g^*(\bb F_{q^m}) \equiv -1 + (-1)^n \# H_g^*(\bb F_{q^m}) \qquad \text{mod } \pi. 
\end{equation}
By (\ref{E: Wan zeta_toric}), $Z(H_g^*, T)^{(-1)^n}$ mod $p$ is a polynomial of degree at most $\binom{d+n}{n}$. Since $p$ is strictly larger than the maximum degrees $W(0)$ and $\binom{d+n}{n}$, equation (\ref{E: equiv}) implies
\[
L(V_g^*, f; T)^{(-1)^{n+1}} \equiv (1- T)^{-1} \cdot Z(H_g^*, T)^{(-1)^n} \equiv det(1 - \psi \circ g^{q-1} T \mid R_d) \qquad \text{mod } \pi.
\]
\end{proof}

\bigskip\noindent
{\bf Remarks.}
\begin{enumerate}
\item Toric exponential sums are those defined over the algebraic torus $\bb G_m^n$. In this case, $g(x) = x_1 \cdots x_n$, and the above shows $L(\bb G_m^n, f; T)^{(-1)^{n+1}}$ has a unique $p$-adic unit root. Adolphson and Sperber \cite{MR2966711} have demonstrated that this unit root has a formula similar to (\ref{E: formula}) using $\c A$-hypergeometric functions.
\item It is quite likely that improvements may be made weakening the condition $p > \max\{ \binom{n + d}{d}+1, D \}$. More interesting would be to obtain formulas for higher congruences. 
\item Through personal communication, Alan Adolphson suggests that $\binom{n + d}{d}$ is generically the precise number of unit roots of $L(V_g^*, f; T)^{(-1)^{n+1}}$. This would likely take the form of an Hasse polynomial (see, for example, \cite{MR2301225}).
\end{enumerate}

\section{Complements of hypersurfaces in affine space}\label{S: many_affine}

Let $f, g \in \bb F_q[x_1, \ldots, x_n]$ be polynomials of degree $D$ and $d$, respectively. Let $H_g$ be the variety in $\bb A^n$ defined by $g = 0$ over $\bb F_q$. Denote by $V_g$ the complement of $H_g$ in $\bb A^n$. For each positive integer $m$ define the exponential sum
\[
S_m(V_g, f) = \sum \psi_m(  f(x) ),
\]
where the sum runs over all $x \in V_g(\bb F_{q^m})$. The associated $L$-function
\[
L(V_g, f; T) := \exp \left( \sum_{m=1}^\infty S_m(V_g, f) \frac{T^m}{m} \right)
\]
is a rational function defined over the cyclotomic field $\bb Q(\zeta_p)$, where $\zeta_p$ is a primitive $p$-th root of unity. 

As in the previous section, let $\hat g$ denote the Teichm\"uller lift of $g$. Set $S := \{1, 2, \ldots, n\}$. For a subset $J \subset S$, denote by $\hat g_J$ the polynomial obtained from $\hat g$ by setting $x_i = 0$ for every $i \in J$. Throughout this section, we will assume that 
\begin{equation}\label{E: Reich2}
\text{for every $J \subset S$, the highest degree homogenous form of $\hat g_J$ is a product of distinct irreducible factors.}
\end{equation}

Let $Z(H_g, T) := \exp \left(\sum N_m \frac{T^m}{m} \right)$ be the zeta function associated with the point count
\[
N_m := \# H_g(\bb F_{q^m}) := \# \{ x \in (\bb F_{q^m})^n \mid g(x) = 0 \}.
\]
As a consequence of the Dwork trace formula, the zeta function has the following simple description mod $p$ (see \cite{MR1650604}). Define the $\bb F_q$-vector space
\[
W_d := (x_1 \cdots x_n \bb F_q[x_1, \ldots, x_n])_{\leq d}
\]
consisting of polynomials over $\bb F_q$ of degree at most $d$ which are divisible by the monomial $x_1 \cdots x_n$. Observe that multiplication by $g^{q-1}$ is a map from $W_d$ into $W_{d q}$, and thus $\psi_q \circ g^{q-1}$ is an operator on $W_d$. Wan \cite{MR1650604} shows that
\begin{equation}\label{E: Wan zeta_affine}
Z(H_g, T)^{(-1)^n} \equiv det(1 - \psi_q \circ g^{q-1} T \mid W_d) \qquad \text{mod } p,
\end{equation}
a polynomial of degree at most $\text{dim}_{\bb F_q} W_d = \binom{d}{n}$. Observe that the Chevalley-Warning result $\# V_g(\bb F_q) \equiv 0$ mod $p$ follows if $n > d$ since $W_d = 0$ in this case.

\begin{theorem}\label{T: main_affine}
Suppose (\ref{E: Reich2}) and $p > \max\{ \binom{n + d}{d}+1, D \}$. Then
\[
L(V_g, f; T) \equiv Z(H_g, T)^{-1} \quad \text{mod}(\pi)
\]
where $\pi := 1 - \zeta_p$. Moreover,
\[
L(V_g, f; T)^{(-1)^{n+1}} \equiv det(1 - \psi_q \circ g^{q-1} T \mid W_d) \quad \text{mod}(\pi),
\]
a polynomial of degree at most $\binom{d}{n}$. In particular, the number of unit roots of $L(V_g, f; T)$ is bounded by $\binom{d}{n}$.
\end{theorem}

\begin{proof}
The toric decomposition of affine space implies
\[
L(V_g, f; T) = \prod_{J \subset S} L(V_{g_J}^*, f_J; T).
\]
Thus, by Theorem \ref{T: main_toric},
\[
L(V_g, f; T) \equiv \prod_{J \subset S} Z( H_{g_J}^*, T)^{-1} \equiv Z(H_g, T)^{-1} \quad  \text{ mod } \pi
\]
The rest of the result follows from (\ref{E: Wan zeta_affine}).
\end{proof}

An immediate corollary of the above is the following Chevalley-Warning type result:

\begin{corollary}\label{C: CW}
Under the same conditions as Theorem \ref{T: main_affine}, suppose further that $n > d$. Then $ord_q S_1(V_g, f) > 0$. 
\end{corollary}

Theorem \ref{T: main_affine} and Corollary \ref{C: CW} raise the following:

\bigskip\noindent
{\bf Questions and Remarks.}
\begin{enumerate}
\item Observe that for a parametrized family of hypersurfaces $g_\lambda$ and regular functions $f_\lambda$ on $V_{g_\lambda}$, dependent on an algebraic parameter $\lambda$, it is expected that the unit roots of $L(V_{g_\lambda}, f_\lambda; T)$ modulo an appropriate uniformizer will only depend on the family $g_\lambda$. For example, setting $g_\lambda = y^2 - x(x-1)(x-\lambda)$, the Legendre family of (affine) elliptic curves mentioned in the introduction, then by (\ref{T: main_affine}) the $L(V_{g_\lambda}, f_\lambda;  T)^{-1}$ will have a unique unit root $\pi_0(\lambda)$ satisfying (\ref{E: formula}) when $H(\lambda) \not= 0$.
\item When does $L(V_g, f; T)$, or equivalently, $Z(V_g, T)$, only have unit roots which are 1-units? This will be the case when $det(1 - \psi_q \circ g^{q-1} T)$ equals $(1-T)^e$ for some $e \leq \binom{d}{n}$. This comes down to looking for rational functions $f/g$ which are eigenvectors of $\psi_q$ with eigenvalue 0 or 1. Can these be characterized? 
\item There have been various improvements to the classical Chevalley-Warning theorem; see for example, \cite{Ax-Zeros_of_polynomials_over_finite_fields} \cite{Katz-On_thm_of_Ax} \cite{MR1215208} \cite{MR1314464}, and in particular, \cite{MR932797}. Are there analogous refinements for Corollary \ref{C: CW}? 
\end{enumerate}

\section{Remark on arithmetic mirror symmetry}\label{S: mirror}

There has been a lot of recent attention paid to studying similarities between zeta functions of mirror pairs of algebraic varieties. A selective list of references includes \cite{Aldi_hucsch}, \cite{Perunicic_2014_Introduction-to-Arithmetic-Mirror-Symmetry}, \cite{MR2454318} \cite{MR2368956} \cite{MR2385244}, \cite{MR2019149}, \cite{MR2672288}, \cite{haessig_mirror}, \cite{MR2774204}, \cite{MR2584833}, \cite{MR2282960}, \cite{MR2194729}, \cite{MR2213683}.  An alternate viewpoint is to study the relation between the complements of these varieties within some fixed ambient spaces. The following is a brief remark in this direction. 

Let $X$ and $Y$ be a strong mirror pair of Calabi-Yau varieties of dimension $d$ defined over $\bb F_q$. We refer the reader to \cite{MR2282960} for the definition of strong mirror pair. Fix projective embeddings $X \hookrightarrow \bb P^n$ and $Y \hookrightarrow \bb P^n$, and denote the complements of $X$ and $Y$ in their respective embeddings by $V_X$ and $V_Y$. Let $f$ and $g$ be regular functions on $V_X$ and $V_Y$, respectively. Define the $L$-function
\[
L(V_X, f; T) := \exp\left( \sum_{m=1}^\infty S_m(V_X, f) \frac{T^m}{m} \right),
\]
where
\[
S_m(V_X, f) := \sum_{\bar x \in V_X(\bb F_{q^m})} \psi_m( f(\bar x)).
\]
Similarly, define $L(V_Y, g; T)$.

\medskip\noindent{\bf Question.} Under suitable conditions on the characteristic $p$, is there a uniformizer $\pi$ such that $L(V_X, f; T) \equiv L(V_Y, g; T)$ mod $\pi$?

A heuristic proof is as follows. Theorems \ref{T: main_toric} and \ref{T: main_affine} suggest that there may exist a uniformizer $\pi$ such that $L(V_X, f; T) \equiv Z(X, T)^{-1}$ mod $\pi$  and $L(V_Y, f; T) \equiv Z(Y, T)^{-1}$ mod $\pi$. The result would then follow from a conjecture of Wan \cite{MR2282960}, which implies that the zeta functions of strong mirror pairs satisfy $Z(X, T) \equiv Z(Y, T)$ mod $q$.

As Wan's conjecture has been proven in the case of the Dwork family of hypersurfaces and its mirror (see \cite{MR2282960}), this question can likely be proven in this case.

\bibliographystyle{amsplain}
\bibliography{../References/References.bib}

\end{document}